\documentclass[12pt,twoside,a4paper]{amsart}
\usepackage{amssymb,hyperref}

\renewcommand \epsilon \varepsilon

\def \O {\mathcal{O}}
\def \C {\mathbb{C}}
\def \N {\mathbb{N}}
\def \R {\mathbb{R}}
\def \local {\O_{\C^n,0}}
\def \locall {\O_{\C^2,0}}
\def \localz {\O_{Z,z}}
\def \Oc {\O_{C,0}}
\def \dbar {\overline{\partial}}
\def \Pres {\dbar(1/P)}
\def \PRES {\dbar\frac 1P}

\DeclareMathOperator{\bs}{bs}
\DeclareMathOperator{\ord}{ord}

\newtheorem{theorem}{Theorem}[section]
\newtheorem{lemma}[theorem]{Lemma}

\theoremstyle{remark}
\newtheorem{remark}[theorem]{Remark}
\numberwithin{equation}{section}

\title[]{The Brian\c con-Skoda number of analytic irreducible planar curves}

\begin{document}

\date{\today}
\author{Jacob Sznajdman}
\address{Department of Mathematics\\Chalmers University of Technology and University of 
Gothenburg\\S-412 96 GOTHENBURG\\SWEDEN}
\email{sznajdma@chalmers.se}

\begin{abstract}
  The Brian\c con-Skoda number of a ring $R$ is defined as the smallest integer k,
  such that for any ideal $I\subset R$ and $l\geq 1$,
  the integral closure of $I^{k+l-1}$ is contained in $I^l$. We compute the Brian\c con-Skoda number of
  the local ring of any analytic irreducible planar curve in terms of its Puiseux characteristics.
  It turns out that this number is closely related to the Milnor number. 

\end{abstract}

\maketitle

\bibliographystyle{amsalpha}

\section{Introduction}\label{intro} 
The Brian\c con-Skoda theorem is a famous theorem in commutative algebra.
It was first proven in 1974 by Jo\"el Brian\c con
and Henri Skoda \cite{BS}. In the original setting, the theorem dealt with the ring of germs
of holomorphic functions at $0 \in \C^n$, but other rings have been studied as well.

Given a ring $R$, one defines the integral closure of an ideal $I\subset R$ as
\begin{align}\label{algebraic_ic}
  \overline{I} = \{\phi\in R :\ \exists (N \geq 1, b_j\in I^j) \ \phi^N + b_1 \phi^{N-1} + \dots + b_N=0\}.
\end{align}
We are interested in integers $k$ such that
the inclusion $\overline{I^{k+l-1}} \subset I^l$ holds
for any ideal $I\subset R$ and $l \geq 1$.  We will denote the smallest such integer $k$ by $\bs(R)$.
If no such integer exists, we say that $\bs(R) = \infty$.
Huneke, \cite{huneke}, proved that $\bs(R)<\infty$,
given some fairly mild assumptions on $R$.
It is desirable to express $\bs(R)$ in terms of invariants of (the singularity of) $R$.

In the case $R$ is the local ring $\localz$ of an analytic variety $Z$
at some point $z$, Huneke's result was later proven analytically in \cite{ass}.
In this setting,
\begin{align}\label{analytic_ic}
  \overline{I} = \{\phi \in \localz:\ \exists K>0\ \ |\phi| \leq K |I| \ \text{on }Z\},
\end{align}
where $|I|= |a_1| + \dots + |a_m|$, and $a_j$ generate $I$.
A proof of \eqref{analytic_ic} is found in \cite{cloture_integral}. We shall prefer the alternative definition
of the Brian\c con-Skoda number as the smallest integer $k$ such that,
for all $\phi \in \localz$ and all ideals $I \subset \localz$,
\begin{align}\label{analytic_premise}
  |\phi| \lesssim |I|^{k+l-1} \ \text{on }Z
\end{align}
implies that $\phi \in I^l$.

In the original setting, that is, $R=\local$, the Brian\c con-Skoda theorem states that $\bs(\local)=n$.
The subject of this paper is the case when $R$ is the local ring at $0$
of an irreducible analytic curve $C$ in $\C^2$. Our main result is a formula that expresses $\bs(C):=\bs(\Oc)$
in terms of the Puiseux characteristics of $C$ at $0$.

Let $m$ be the multiplicity of the curve $C$ at the origin.
Recall that $m=1$ if and only if the curve is smooth near $0$. Note that
$\bs(C)=1$ if and only if $C$ is smooth;
if $\bs(C)=1$, then every weakly holomorphic on $C$ is strongly holomorphic, so $C$ is normal, and therefore smooth.
The same steps applied backwards give that $\bs(C)=1$ whenever $C$ is smooth.

When $\dim Z > 1$, one may ask if
$Z$ is regular if and only if $\bs(Z) = \dim Z$. We do not know the answer to this question.
However, it is known that there are rings with mild singularities
so that $\bs(R) = \dim R$, for example, the class of pseudo-rational rings, see \cite{lipman_tessier}.

According to Puiseux's theorem, for a suitable choice of coordinates, the curve
is locally parametrized by the normalization $\Pi : \Delta \subset \C \to C$,
given by $(z,w) = \Pi(t) = (t^m,g(t))$, for some analytic function $g(t) = \sum_{k\geq m} c_k t^k$.
Puiseux's theorem is proven analytically in \cite{whitney}, Chapter 1, Section 10, and algebraically in \cite{lefschetz}.

Set $e_0 = m$ and
define inductively
\begin{align}\label{beta_def}
  \beta_j = \min \{k\in \N: c_k \neq 0, e_{j-1} \nmid k\}
\end{align}
and
\begin{align*}
  e_j = \gcd(e_{j-1},\beta_j) = \gcd(m,\beta_1,\beta_2,\ldots,\beta_j).
\end{align*}
The construction stops when, for some integer $M$, one has $e_M=1$.
These numbers are known as the Puiseux characteristics of the curve. Since $c_k=0$ for $k<m$, one has
$\beta_j \geq m$. It is not hard to see that $\beta_\bullet$ is strictly increasing and $e_\bullet$ is
strictly decreasing. Thus $\beta_j \geq \beta_1 > m$.

Recall that $\lceil x \rceil$ is the ceiling function applied to $x$,
that is, the smallest integer $n$ such that $n\geq x$.

\begin{theorem}\label{main}
  For any germ of an analytic irreducible planar curve $C$, one has
  \begin{align*}
    \bs(C) = \Bigg\lceil\frac 1m (1 +\sum_{i=1}^M (e_{i-1}-e_i)\beta_i)\Bigg\rceil.
  \end{align*}
\end{theorem}

We see from this formula and the comments above that indeed $\bs(C)=1$ if and only if $C$ is smooth.
By Remark 10.10 in \cite{milnor}, this formula can be rewritten as
\begin{align*}
  \bs(C) = \bigg\lceil 1+ \frac \mu m \bigg\rceil,
\end{align*}
where $\mu$ is the Milnor number of the germ $C$.

\section{Analytic formulation of the Brian\c con-Skoda problem}\label{the_proof}

Any ideal $I=(a_1,\ldots,a_m)\subset\Oc$ has a reduction, that is, an ideal $J\subset I$ such that $|J|\simeq|I|$,
and $J=(a)$ is generated by $\dim C~=~1$ element. This is not hard to see. Indeed,
if $\Pi : \C \to C$ is the normalization of $C$,
we see that $|a| \simeq |I|$ holds if and only if 
\begin{align*}
  |\Pi^* a| \lesssim |\Pi^*a_1| + \dots +|\Pi^*a_m|.
\end{align*}
Clearly $|\Pi^*a_1| + \dots +|\Pi^*a_m| \simeq |\Pi^*a_j|$,
where $j$ is an index such that the vanishing order of $\Pi^*a_j$ at $0$ is minimal.
Thus $J$ is a reduction of $I$ if we take $a=a_j$. For the purposes of finding $\bs(C)$, we can replace
$I$ by its reduction $J$. We henceforth set $I=(a)$.

A (germ of a) meromorphic form on $C$ is defined as a meromorphic form on $C_{reg}$ which is
the pull-back of a meromorphic form near $0\in \C^2$ with respect to the inclusion map $i: C \to \C^2$.
A weakly holomorphic function on $C$ is a holomorphic function on $C_{reg}$ that is locally bounded
on $C$. Any weakly holomorphic function is meromorphic.

The following lemma gives an alternative characterization of the number $\bs(C)$.

\begin{lemma}\label{weak_problem}
  The Brian\c con-Skoda number $\bs(C)$ is the smallest integer $k\geq 1$,
  such that if we are given any $a\in \Oc$ with $a(0)=0$, 
  and a weakly holomorphic function $\psi$, then
  \begin{align}\label{weak_premise}
    |\psi| \lesssim |a|^{k-1}  \quad \text{on } C,
  \end{align}
  implies that $\psi$ is strongly holomorphic.
\end{lemma}
\begin{proof}
  Assume that $k$ is any integer that satisfies the property given in the lemma.
  Take any non-trivial ideal $(a) \subset \Oc$ and a function $\phi$ so that $|\phi| \lesssim |a|^{k+l-1}$.
  Then $\psi = \phi/a^l$ is meromorphic and, satisfies \eqref{weak_premise}.
  In particular, $\psi$ is weakly holomorphic. By our assumption on $k$, $\psi$ is strongly holomorphic,
  so $\phi \in (a)^l$. We have thus shown that $\bs(C) \leq k$.
  
  It remains to show that $k=\bs(C)$ has the property given in the lemma. Take any weakly holomorphic function
  $\psi=f/g$ on $C$, such that $f,g \in \Oc$ and \eqref{weak_premise} holds.
  If $g$ is a unit, there is nothing to prove, and otherwise we choose $a=g$.
  Then clearly $|f| \lesssim |g|^k$. The (alternative) definition
  of $\bs(C)$, cf. \eqref{analytic_premise}, now gives that $f \in (g)$, so $\psi$ is indeed strongly holomorphic.\qed
\end{proof}

We need some preliminaries before stating a criterion for when a weakly holomorphic function
is strongly holomorphic.
A $(p,q)$-current $T$ on $C$ is a current acting on $(1-p,1-q)$-forms in the ambient space,
with the additional requirement that $T.\xi = 0$ whenever $i^* \xi = 0$ on $C_{reg}$.
The $\dbar$-operator is defined as usual by $\dbar T.\xi = (-1)^{p+q+1}T.\dbar \xi$, where $\xi$ is a test form.
If $i^* \xi = 0$ on $C_{reg}$, then $i^* \dbar \xi = \dbar i^* \xi = 0$ on $C_{reg}$. Thus $\dbar T$ is
a $(p,q+1)$-current on $C$.
Any meromorphic $(p,q)$-form $\eta$ on $C$ can be seen as a $(p,q)$-current on $C$ which acts by
\begin{align*}
  \eta.\xi = \int_C \eta \wedge \xi := \int_{\C} \Pi^* (\eta \wedge \xi),
\end{align*}
where the right-most side is a principal value integral of a meromorphic form in one variable.

According to Weierstrass preparation theorem, for an appropriate choice of coordinates, we may assume that $C$
is the zero locus of a Weierstrass polynomial $P(z,w) = w^m + b_1(z) w^{m-1} + \dots + b_m(z)$, where
$m$ was defined as the multiplicity of $C$ at $0$.
Let $\omega'$ be any meromorphic form acting on the tangent space of $\C^2$, but defined
on $C$, such that
\begin{align}\label{leray_premise}
  dP \wedge \omega' = dz\wedge dw.
\end{align}
Then $\omega = i^* \omega'$ is a well-defined meromorphic form on $C$.
We choose the representative
\begin{align}\label{omega_def}
  \omega' = -\frac 1{P'_w}dz
\end{align}
for $\omega$; this shows that \eqref{leray_premise} can be satisfied.
Theorem~\ref{criterion2} below is a reformulation of a result by A. Tsikh, \cite{tsikh}.
Tsikh's proof, which is given in Section~\ref{criterion_proof}, relies on residue theory in two variables,
and implicitly, therefore also Hironaka resolutions. 
We are not aware of any elementary proof.

\begin{theorem}\label{criterion2}
  Let $\psi$ be any meromorphic function on $C$. Then $\psi$ is strongly holomorphic if and only if $\psi \omega$
  is $\dbar$-closed.
\end{theorem}

\begin{remark}\label{gen_remark}
  The form $\omega$ generates all $\dbar$-closed meromorphic $(1,0)$-forms on $C$. In fact, let $\xi$ be any meromorphic $(1,0)$-form on
  $C$. Then there is a meromorphic function $\alpha$ such that $\xi = \alpha \omega$. Hence, if $\dbar \xi =0$,
  we have that $\alpha \omega$ is $\dbar$-closed, so by Theorem~\ref{criterion2}, $\alpha \in \Oc$.
\end{remark}

\begin{lemma}\label{inexplicit_lma}
  The Brian\c con-Skoda number is given by the identity
  \begin{align}\label{inexplicit_bs}
    \bs(C) = \big\lceil(1 + \ord_0(\Pi^* P'_w))/m \big\rceil.
  \end{align}
\end{lemma}
\begin{proof}
  Let $k$ be the right hand side of \eqref{inexplicit_bs}.
  We will first show that $\bs(C) \leq k$.
  Assume that $\psi$ is weakly holomorphic on $C$
  and satisfies \eqref{weak_premise}.
  According to Lemma~\ref{weak_problem}, it suffices to show that $\psi$ is strongly holomorphic.

  Using $\Pi(t)=(t^m,g(t))$ and \eqref{weak_premise}, we see that
  \begin{align}\label{suffices}
    -\psi \omega.\dbar \eta = \int_C \frac{\psi}{P'_w} dz \wedge \dbar \eta =
    \int_{\C} \frac {h(t) t^{(k-1)\ord_0(\Pi^*a)}}{t^{\ord_0(\Pi^*P'_w)}}d(t^m)\wedge \dbar(\Pi^* \eta),
  \end{align}
  where $h$ is holomorphic in $t$ and $\eta\in C^{\infty}_0(\C^2,0)$ is an arbitrary test function.
  In view of Theorem~\ref{criterion2}, we would like to show that $\psi \omega.\dbar \eta = 0$,
  because then $\psi$ is holomorphic.
  The last integral in equation~\eqref{suffices} vanishes if
  \begin{align}\label{deriving_inexplicit}
    (k-1)\ord_0(\Pi^*a) \geq \ord_0(\Pi^*P'_w)-(m-1),
  \end{align}
  since then the integrand is the product of a holomorphic function
  and a $\dbar$-exact form with compact support.
  Clearly the worst case in \eqref{deriving_inexplicit} occurs when $\ord_0(\Pi^*a)$ is minimal. Since $\ord_0(g(t))\geq m$, the minimal value
  is $\ord_0(\Pi^*a) = m$, which is attained for example when $a=z$. By the definition of $k$, we see that
  \eqref{deriving_inexplicit} holds with equality in the worst case.

  We will now show that $\bs(C) \geq k$.
  Let $Q = (1 + \ord_0(\Pi^* P'_w))/m$, so that $k = \lceil Q \rceil$.
  Then
  \begin{align*}
    |\Pi^* P'_w| \simeq |t|^{\ord_0(\Pi^* P'_w)} = |\Pi^* z|^{\ord_0(\Pi^* P'_w)/m} = |\Pi^* z|^{Q-1/m},
  \end{align*}
  and thus
  $|P'_w| \simeq |z|^{Q-1/m}$. Let us check that $P'_w \notin (z)$, that is, that $P'_w / z \notin \Oc$.
  If $m=1$, then $\bs(C)=1$ and we have nothing to prove, so assume that $m>1$. Recall that $e_0 > e_1$,
  so we have
  \begin{align*}
    \ord_0(\Pi^*P'_w)~\geq~\beta_1~>~m~=~\ord_0(\Pi^* z).
  \end{align*}
  Thus $\psi:= P'_w /z$ is weakly holomorphic on $C$. If $\xi$ is a test function on $C$ such that $\xi(0)\neq 0$, then
  \begin{align*}
    \int_C \psi \omega \wedge \dbar \xi = - \int_C \frac{dz}{z} \wedge \dbar \xi = -m \int_{\C} \frac{dt}{t}\wedge \dbar \Pi^* \xi =-2m\pi i \xi(0)\neq 0.
  \end{align*}
  Theorem~\ref{criterion2} now gives that $\psi$ is not strongly holomorphic on $C$.
  Thus $\bs(C) > Q-1/m$. Since $Q \in 1/m \cdot \mathbb{Z}$, $\bs(C) = \lceil Q \rceil$ as claimed.\qed
\end{proof}

\begin{remark}\label{real_bs}
  One may wish to remove the restriction on the Brian\c con-Skoda number that it has to be an integer.
  It is then natural to consider
  \begin{align}\label{real_bs_def}
    \kappa := \inf \{k \in \R: |\psi|\lesssim |a|^{k-1} \implies \psi \in \Oc\},
  \end{align}
  where the implication is assumed to hold for all weakly holomorphic functions $\psi$ on $C$ and all $a\in \Oc$ that are not invertible.
  The argument below yields that the set in the right hand side of \eqref{real_bs_def} is open.
  We claim that $\kappa = Q-1/m$. This can be seen as follows.
  Note that the same infimum is attained in \eqref{real_bs_def} if we assume that $a=z$ (or more generally that $\ord_0 \Pi^* a = m$).
  The example $\psi = P'_w / z$ which we considered before, shows that $\kappa$ cannot be smaller than $Q-1/m$.
  If $k=Q-1/m+\epsilon$ and $|\psi|\lesssim |a|^{k-1}$, then $\ord_0 \Pi^* \psi > m(k-1)$.
  Hence
  \begin{align*}
    \ord_0 \Pi^* \psi \geq m(k-1)+1 = m(Q-1) = \ord_0(\Pi^* P'_w) - (m-1),
  \end{align*}
  but then
  the first half of the proof of Lemma~\ref{inexplicit_lma} shows that $\psi\in\Oc$.
  We conclude that $k=Q-1/m+\epsilon$ is a candidate for the infimum in \eqref{real_bs_def}, so $\kappa \leq Q-1/m$
  since $\epsilon$ is arbitrary.
  
\end{remark}

In Section~\ref{singularity} we express $\ord_0(\Pi^* P'_w)$ in terms of Puiseux's invariants in Lemma~\ref{order_lemma}.
Together with Lemma~\ref{inexplicit_lma}, this yields Theorem~\ref{main}.

\section{Proof of Theorem 2.2}\label{criterion_proof}

One can define the principal value current $1/P$ as follows:
Let $\chi$ be a smooth cut-off function such that $\chi \equiv 0$ on some interval $[0,\delta]$ and
$\chi \equiv 1$ on $[1,\infty)$. For any full-degree test form $\xi$, one defines
  \begin{align}\label{principal_value}
    \int_{\C^n} \frac 1P \xi = \lim_{\epsilon \to 0} \int_{\C^n} \chi(|P|/\epsilon) \frac{\xi}{P}.
  \end{align}
  The existence of such principal values was proved in \cite{herreralieberman}, although with a slightly
  different definition; however \eqref{principal_value} is just an avarage of principal values
  in the sense of Herrera-Lieberman. We now apply the $\dbar$-operator in the sense of currents to obtain $\dbar (1/P)$.

  Let $\psi$ be a meromorphic function on $C$ and let $\Psi=\Psi_1/\Psi_2$ be a representative in the ambient space
  such that $\Psi_1$ and $\Psi_2$ are relatively prime.
  Then $\psi \dbar (1/P)$ can be defined as
  \begin{align}\label{prod_def}
    \lim_{\epsilon \to 0} \Psi_1 \frac{\chi(|\Psi_2|/\epsilon)}{\Psi_2} \dbar \frac 1P.
  \end{align}
  It is not obvious that \eqref{prod_def} is a valid definition, that is, that
  the limits exist and does not depend on the choice of $\chi$, nor on the representative $\Psi$ of $\psi$;
  see e.g. Theorem~1 in \cite{bjork_samuelsson} and the comments that follow it.

  We now formulate
  a criterion which is due to Tsikh, \cite{tsikh}. 
  A generalized version of this criterion can be found in \cite{ma_criterion}.

  \begin{theorem}\label{criterion}
    Let $\psi$ be any meromorphic function on $C$. Then $\psi$ is strongly holomorphic if and only if $\psi \dbar(1/P)$
    is $\dbar$-closed.
  \end{theorem}
  \begin{proof}
    Assume that $\psi \dbar(1/P)$ is $\dbar$-closed and write $\psi=g/h$ for some functions $g,h\in \locall$.
    By basic rules for Coleff-Herrera products, which can be deduced for example from
    Theorem~1 in \cite{bjork_samuelsson}, we have
    \begin{align*}
      g \dbar\frac 1h \wedge \dbar\frac 1P=\dbar (\psi \dbar \frac 1P)=0.
    \end{align*}
    By the duality theorem, \cite{passare}, \cite{DS},
    it follows that $g\in(h,P)$,
    so $g-\alpha h\in (P)$ for some $\alpha \in \locall$. Thus $\psi=\alpha$ on $C$.
    The converse is immediate.\qed
  \end{proof}

  The form $\omega$ is by definition Leray's residue form of $dz\wedge dw /P$, and
  Leray's residue formula states that
  \begin{align*}
    \int_C \omega \wedge \xi = \frac 1{2\pi i} \int \PRES \wedge dz\wedge dw \wedge \xi
  \end{align*}
  for any $(0,1)$-test form $\xi$.
  It follows from this formula that $\psi \omega$ is $\dbar$-closed
  if and only if $\psi \Pres$ is $\dbar$-closed.
  Therefore, Theorem~\ref{criterion2} follows from Theorem~\ref{criterion}.

  \section{The singularity of Leray's residue form}\label{singularity}
  In this section, we will prove Lemma~\ref{order_lemma} below, and thereby finish the proof of Theorem~\ref{main}.
  Previously we have chosen the coordinates $(z,w)$ in $\C^2$, so that $C$ is the
  zero locus of an irreducible Weierstrass polynomial
  \begin{align*}
    P(z,w) &= w^m + a_1(z) w^{m-1} + \dots + a_m(z).
  \end{align*}

  \begin{lemma}\label{order_lemma}
    The vanishing order of $\Pi^* P'_w$ at $0$ is
    \begin{align}\label{order_formula}
      \sum_{l=1}^M (e_{l-1} - e_l)\beta_l.
    \end{align}
  \end{lemma}

  \begin{proof}
    For any (small) $z\neq 0$, let $\alpha_j(z)$, $1\leq j\leq m$, be the roots of $w \mapsto P(z,w)$;
    then $P = \prod_{j=1}^m (w-\alpha_j(z))$. It is well-known that the $\alpha_j(z)$ are holomorphic on 
    some sufficiently small neighbourhood $V$~of~$z$. Recall that $\Pi(t) = (t^m,g(t))$, where $g(t) = \sum_{k\geq m}c_kt^k$.
    Choose $t$ such that $t^m=z$, and let $\rho$ be a primitive $m$:th root of unity.
    Then $\Pi$ maps $\{t, \rho t, \ldots, \rho^{m-1}t\}$ bijectively onto the fibre $(\{z\}\times \C_w) \cap C$.
    Thus $\{g(\rho^j t): 1 \leq j \leq m\}$ are the roots of $w \mapsto P(z,w)$. After possibly
    possibly renumbering the $m$ roots, we get $g(\rho^j t) = \alpha_j(t^m)$.
    
    We claim that
    \begin{align}\label{P_claim}
      \Pi^* P'_w = \prod_{j=1}^{m-1} (g(t)-g(\rho^{j}t)).  
    \end{align}
    Since we are asserting the identity of two holomorphic functions, it is enough to prove equality locally
    outside of $\{t=0\}$. We have
    \begin{align*}
      P'_w = \sum_{l=1}^m\prod_{1\leq j \leq m,\ j\neq l} (w-\alpha_j(z)),
    \end{align*}
    so
    \begin{align*}
      \Pi^* P'_w = \sum_{l=1}^m\prod_{1\leq j \leq m,\ j\neq l} (g(t)-\alpha_j(t^m))
      = \prod_{1\leq j \leq m-1} (g(t)-g(\rho^jt)).
    \end{align*}

    Since $g(t) = \sum_{k=m}^\infty c_k t^k$, we have that
    \begin{align}\label{g_diff}
      g(t) - g(\rho^j t) = \sum_{k=m}^\infty c_k (1-\rho^{kj})t^k.
    \end{align}
    
    Let $k_j^* = \ord_0(g(t) - g(\rho^j t))$. 
    Then by \eqref{g_diff},
    \begin{align}\label{kjstar_appear}
      k_j^* &= \min \{k : c_k \neq 0, (1-\rho^{kj}) \neq 0\} = \\\notag
      &= \min \{k : c_k \neq 0, m \nmid kj\}.
    \end{align}
    We also consider the number
    \begin{align*}
      r(j) = \min \{l : m \nmid j\beta_l\},
    \end{align*}
    for each $1\leq j \leq m-1$.
    The sequence $\beta_1 , \beta_2 , \ldots , \beta_M$ is strictly increasing, so
    \begin{align}\label{trams}
      m \mid j\beta_l \quad\text{for all}\quad l<r(j),
    \end{align}
    and
    \begin{align}\label{trams2}
      m \nmid j\beta_{r(j)}.      
    \end{align}
    Since $je_l=\gcd(jm,j\beta_1,\ldots,j\beta_l)$,
    the statements \eqref{trams} and \eqref{trams2} together imply
    \begin{align}\label{r_prop}
      m \mid je_l \quad \text{ if and only if} \quad l < r(j).
    \end{align}

    Now note that $k_j^* \leq \beta_{r(j)}$.
    We claim that, in fact, $k_j^* = \beta_{r(j)}$. Assume to the contrary that $k_j^* < \beta_{r(j)}$.
    Then $e_{r(j)-1} \mid k_j^*$ by \eqref{beta_def}, so we have
    \begin{align*}
      je_{r(j)-1} \mid jk_j^*.
    \end{align*}
    Together with \eqref{r_prop}, this gives $m \mid jk_j^*$, contradicting the definition of $k_j^*$. 

    By \eqref{P_claim}, we have
    \begin{align}\label{value_sum}
      \ord_0(\Pi^* P'_w) &= \sum_{j=1}^{m-1} \ord_0(g(t) - g(\rho^j t)) = \\\notag
      &=\sum_{j=1}^{m-1} k^*_j = \sum_{l=1}^M \#\{j: k^*_j = \beta_l\}\beta_l,
    \end{align}
    where the last equality follows since $k_j^* = \beta_{r(j)}$.

    Using that $k_j^* = \beta_{r(j)}$ and the strict monotonicity of the sequence $\beta_1 , \beta_2 , \ldots$,
    we see that $k^*_j \geq \beta_l$ is equivalent to $r(j) \geq l$. Thus \eqref{r_prop} gives that
    \begin{align*}
      \#\{j: k^*_j \geq \beta_l\} &= \#\{j\in [1,m-1]: m \mid je_{l-1}\}=\\
      &=\#\{\frac{m}{e_{l-1}},2\frac{m}{e_{l-1}},\ldots,(e_{l-1}-1)\frac{m}{e_{l-1}}\}=e_{l-1}-1.
    \end{align*}
    Clearly, $\#\{j: k^*_j = \beta_l\} = \#\{j: k^*_j \geq \beta_l\} - \#\{j: k^*_j \geq \beta_{l+1}\} = e_{l-1}-e_l$.
    We substitute this into \eqref{value_sum}, and thereby obtain the desired formula \eqref{order_formula}.\qed
  \end{proof}

  \begin{remark}
    Recall that $\omega = - i^*[(P'_w)^{-1}dz]$, where $i$ is the inclusion of $C\setminus \{0\}$ into $\C^2$.
    It follows from Lemma~\ref{order_lemma} and Remark 10.10 in Milnor's book \cite{milnor} that $\Pi^* \omega =u(t)t^{-\mu} dt$,
    where $u$ is holomorphic and non-vanishing and $\mu$ is the Milnor number of $C$. By Remark~\ref{gen_remark}, we then have
    that the maximal singularity of any $\dbar$-closed $(1,0)$-form on $C$ is precisely the Milnor number.
  \end{remark}

\end{document}